\documentclass[12pt,reqno]{amsart}
\usepackage{amsmath,amssymb,amsfonts,amscd,latexsym,amsthm,mathrsfs}
\usepackage[unicode]{hyperref}
\let\eps\varepsilon
\let\ol\overline

\newcommand{\M}{{\mathrm M}}
\newcommand{\house}[1]{\hbox{\vrule width.4pt%
 \vbox{{\hrule height.4pt}\vskip2.5pt\hbox{\,$#1$\,}}\vrule width.4pt}}

\newtheorem{theorem}{Theorem}
\newtheorem{question}{Question}
\begin{document}

\title{Hedgehogs in Lehmer's problem}

\author{Jan-Willem M. van Ittersum}
\address{Mathematisch Instituut, Universiteit Utrecht, Postbus 80.010, 3508~TA Utrecht, Netherlands}
\address{
Max-Planck-Institut f\"ur Mathematik, Vivatsgasse 7, 53111 Bonn, Germany}
\email{j.w.m.vanittersum@uu.nl}

\author{Berend Ringeling}
\address{Department of Mathematics, IMAPP, Radboud University, PO Box 9010, 6500~GL Nijmegen, Netherlands}
\email{b.ringeling@math.ru.nl}

\author{Wadim Zudilin}
\address{Department of Mathematics, IMAPP, Radboud University, PO Box 9010, 6500~GL Nijmegen, Netherlands}
\email{w.zudilin@math.ru.nl}

\dedicatory{To Gunther Cornelissen, with warm wishes, on the nice ocassion of reaching the age\\(for the first time!) that can be written as a sum of two positive squares in two different ways.\\Niet elke egel is stekelig!}

\begin{abstract}
Motivated by a famous question of Lehmer about the Mahler measure we study and solve its analytic analogue.
\end{abstract}

\date{31 May 2021}

\thanks{The work of the second author is supported by NWO grant OCENW.KLEIN.006.}

\maketitle

\section{Introduction}
\label{sec1}

Several deep arithmetic questions are known about polynomials with integer coefficients.
One of them raised by Lehmer in the 1930s asks, for a monic irreducible polynomial $P(x)=\prod_{j=1}^d(x-\alpha_j)\in\mathbb Z[x]$, whether the quantity $\M(P(x))=\prod_{j=1}^d\max\{1,|\alpha_j|\}$ can be made arbitrary close to but larger than~1.
The characteristic $\M(P(x))$ is known as the Mahler measure \cite{BZ20};
in spite of the name coined after Mahler's works in the 1960s, many results about it are rather classical.
One of them, due to Kronecker, says that $\M(P(x))=1$ if and only if $P(x)=x$ or the polynomial is cyclotomic, that is, all its zeros are roots of unity.

A related question, usually considered as a satellite to Lehmer's problem, about the so-called house of a non-zero algebraic integer $\alpha$ defined through its minimal polynomial $P(x)\in\mathbb Z[x]$ as $\house{\alpha}=\max_j|\alpha_j|$, was asked by Schinzel and Zassenhaus in the 1960s and answered only recently by Dimitrov \cite{Di19}.
He proved that $\house{\alpha}\ge2^{1/(4d)}$ for any non-zero algebraic integer $\alpha$ which is not a root of unity; the latter option clearly corresponds to $\house{\alpha}=1$.

Dimitrov's ingenious argument transforms the arithmetic problem into an analytic one.
In this note we discuss potentials of Dimitrov's approach to Lehmer's problem.

\section{Principal results}
\label{sec2}

Consider a monic irreducible \emph{non-cyclotomic} polynomial $P(x)=\prod_{j=1}^d(x-\alpha_j)\in\mathbb Z[x]$ of degree $d>1$ and assume that the polynomial $\prod_{j=1}^d(x-\alpha_j^2)\in\mathbb Z[x]$ is irreducible as well. (Otherwise the Mahler measure of $P(x)$ is bounded from below through the measures of irreducible factors of the latter polynomial.)
As in~\cite{Di19}, Dimitrov's cyclotomicity criterion together with Kronecker's rationality criterion and a theorem of P\'olya imply that the \emph{hedgehog}
$$
K=K(\beta_1,\dots,\beta_n)=\bigcup_{k=1}^n[0,\beta_j]
=\bigcup_{j=1}^d[0,\alpha_j^2]\cup\bigcup_{j=1}^d[0,\alpha_j^4],
$$
whose spines originate from the origin and end up at $\alpha_j^2,\alpha_j^4$ for $j=1,\dots,d$,
has (logarithmic) capacity (\emph{aka} transfinite diameter) $t(K)$ at least~1.
Then Dubinin's theorem \cite{Du85} applies, which claims that $t(K)\le4^{-1/n}\max_j|\beta_j|$ (with the equality attained if and only if the hedgehog $K$ is rotationally symmetric), and produces the estimate for $\house{\alpha_1}=\bigl(\max_j|\beta_j|\bigr)^{1/4}$ since $n\le2d$.

When dealing with Lehmer's problem instead, one becomes interested in estimating the `Mahler measure of hedgehog', namely the quantity $\prod_{j=1}^n\max\{1,|\beta_j|\}$, because any non-trivial (bounded away from~1) absolute estimate for it would imply a non-trivial estimate for the Mahler measure of~$P(x)$.
In this setting, Dubinin's theorem only implies the estimate $\prod_{j=1}^n\max\{1,|\beta_j|\}\ge4^{1/n}$ for a hedgehog of capacity at least~1, which depends on~$n$.
The Mahler measure of the rotationally symmetric hedgehog on $n$ spines, which is optimal in Dubinin's result, is equal to~4 (thus, independent of~$n$), which certainly loses to the Mahler measure $1.91445008\dots$ of the `Lehmer hedgehog' attached to the polynomial $x^{10}+x^9-x^7-x^6-x^5-x^4-x^3+x+1$ but also to the measure $3.07959562\dots$ of hedgehog constructed on Smyth's polynomial $x^3-x-1$.
The following question arises in a natural way.

\begin{question}
\label{qu1}
What is the minimum of $\prod_{j=1}^n\max\{1,|\beta_j|\}$ taken over all hedgehogs $K=K(\beta_1,\dots,\beta_n)$ of capacity at least~$1$\textup?
\end{question}

Notice that answering this question for hedgehogs of capacity exactly $1$ is sufficient, since the capacity satisfies $t(K_1)\le t(K_2)$ for any compacts $K_1\subset K_2$ in~$\mathbb C$.

In order to approach Question~\ref{qu1} we use a different construction of hedgehogs outlined in the post of Eremenko to the question in~\cite{Le11} with details exposed in~\cite{Sc20}.
Any hedgehog $K=K(\beta_1,\dots,\beta_n)$ of capacity \emph{precisely} $1$ is in a bijective correspondence (up to rotation!) with the set of points $z_1,\dots,z_n$ on the unit circle with prescribed \emph{positive} real weights $r_1,\dots,r_n$ satisfying $r_1+\dots+r_n=1$.
Namely, the mapping
$$
F(z)=\prod_{k=1}^n\big((z-z_k)(z^{-1}-\ol z_k)\big)^{r_k}
$$
is a Riemann mapping of the complement of the closed unit disk to the complement $\hat{\mathbb C}\setminus K$ of hedgehog.
It is not easy to write down the corresponding $\beta_j$ explicitly but for their absolute values we get
$$
|\beta_j|=\max_{z\in[z_{j-1},z_j]}|F(z)|
=\max_{z\in[z_{j-1},z_j]}\prod_{k=1}^n|z-z_k|^{2r_k}
\quad\text{for}\; j=1,\dots,n,
$$
where we take conventionally $z_0=z_n$ and understand $[z_{j-1},z_j]$ as \emph{arcs} of the unit circle.
It means that if $C\ge1$ is the minimum of
$$
\prod_{j=1}^n\max\bigg\{1,\max_{z\in[z_{j-1},z_j]}\prod_{k=1}^n|z-z_k|^{r_k}\bigg\}
$$
taken over all $n$ and all possible weighted configurations $z_1,\dots,z_n$, then $C^2$ is the minimum in Question~\ref{qu1}.

Furthermore, in the spirit of \cite{KL02} observe that from the continuity considerations it suffices to compute the required minimum $C$ for \emph{rational} positive weights $r_1,\dots,r_n$.
Assuming the latter and writing $r_j=a_j/m$ for positive integers $a_1,\dots,a_n$ and $m=a_1+\dots+a_n$, we are for the $m$th root of the minimum of
$$
\prod_{j=1}^n\max\bigg\{1,\max_{z\in[z_{j-1},z_j]}\prod_{k=1}^n|z-z_k|^{a_k}\bigg\}
=\prod_{j=1}^m\max\bigg\{1,\max_{z\in[z_{j-1}',z_j']}\prod_{k=1}^m|z-z_k'|\bigg\},
$$
where $z_1',z_2',\dots,z_m'$ is the multi-set
$$
\underbrace{z_1,\dots,z_1}_{a_1 \;\text{times}}, \;
\underbrace{z_2,\dots,z_2}_{a_2 \;\text{times}}, \; \dots, \;
\underbrace{z_n,\dots,z_n}_{a_n \;\text{times}}
$$
with prescribed weights all equal to~1.
This means that it is enough to compute the minimum for the case of equal weights, $r_1=\dots=r_n=1/n$, and we may give the following alternative formulation of Question~\ref{qu1}.

\begin{question}
\label{qu2}
What is the minimum $C_n$ of
$$
\prod_{j=1}^n\max\bigg\{1,\max_{z\in[z_{j-1},z_j]}\prod_{k=1}^n|z-z_k|\bigg\}^{1/n}
$$
taken over all configurations of points $z_1,\dots,z_n$ on the unit circle $|z|=1$\textup?
The points are not required to be distinct and $[z_{j-1},z_j]$ is understood as the corresponding arc of the circle, $z_0$ is identified with~$z_n$.
\end{question}

Though there is no explicit requirement on the order of precedence, the minimum corresponds to the successive location of $z_1,\dots,z_n$ on the circle.

A comparison with Dubinin's result suggests that good candidates for the minima in Question~\ref{qu2} may originate from configurations, in which all factors in the defining product but one are equal to~$1$.
In our answer to the question we show that this is essentially the case by computing the related minima $C_n^*$ explicitly.

\begin{theorem}
\label{th1}
For the quantity $C_n$ we have the inequality $C_n\le C_n^*$,
where $C_n^*=\big(T_n(2^{1/n})\big)^{1/n}$ and
$$
T_n(x)
=\sum_{k=0}^{\lfloor n/2\rfloor}\binom{n}{2k}(x^2-1)^kx^{n-2k}
$$
denotes the $n$th Chebyshev polynomial of the first kind.
\end{theorem}

\begin{theorem}
\label{th2}
For the quantity $C_n^*$ in Theorem~\textup{\ref{th1}} we have the asymptotic expansion
$$
C_n^*=1 + \nu - \frac14\nu^3 + \frac5{96}\nu^5 - \frac1{128}\nu^7 + O(\nu^9)
$$
in terms of $\nu=\sqrt{(\log4)/n}$, as $n\to\infty$. In particular, $(C_n^*)^{\sqrt n}\to e^{\sqrt{\log4}}$ and $C_n^*\to1$ as $n\to\infty$.
\end{theorem}

Thus, our results imply that the minimum in Question~\ref{qu1} is equal to~$1$, meaning that an analogue of Lehmer's problem in an analytic setting is trivial.
This brings no consequences to Lehmer's problem itself, as we are not aware of a recipe to cook up polynomials in $\mathbb Z[x]$ from optimal (or near optimal) configurations of $z_1,\dots,z_n$ on the unit circle.

\section{Proofs}
\label{sec3}

\begin{proof}[Proof of Theorem~\textup{\ref{th1}}]
We look for a configuration of the points $z_1,\dots,z_n$ on the unit circle such that the maximum of $|Q(z)|$, where $Q(z)=(z-z_1)\dotsb(z-z_n)$, on all the arcs $[z_{j-1},z_j]$ but one is equal to~$1$:
$$
\max_{z\in[z_{j-1},z_j]}|Q(z)|=|Q(z_j^*)|=1 \quad\text{for}\; z_j^*\in(z_{j-1},z_j), \quad\text{where}\; j=2,\dots,n.
$$
At the same time, the $k$th Chebyshev polynomial $T_k(x)=2^{k-1}x^k+\dotsb$ is known to satisfy $|T_k(x)|\le1$ on the interval $-1\le x\le1$, with \emph{all} the extrema on the interval to be either $-1$ or~$1$.
Note that $T_k(x)$ has $k$ distinct real zeroes on the open interval $-1<x<1$ and satisfies $T_k(1)=(-1)^kT_k(-1)=1$.
Therefore, setting for $n=2k$ even,
\begin{equation*}
Q(z)=z^k\,T_k\Big(2^{1/k}\Big(\frac{z+z^{-1}}2-1\Big)+1\Big),
\end{equation*}
we get a \emph{monic} polynomial of degree $n$ with the desired properties; its zeroes $z_1,\dots,z_n$ ordered in pairs, $z_{n-j}=\ol z_j=z_j^{-1}$ for $j=1,\dots,k$, correspond to the real zeroes $2^{1/k}\big((z_j+z_j^{-1})/2-1\big)+1$ of the polynomial $T_k(x)$ on the interval $-1<x<1$.
Then
\begin{align*}
\max_{z\in[z_n,z_1]}|Q(z)|=\max_{|z|=1}|Q(z)|=|Q(-1)|
=|T_k(1-2^{1+1/k})|&=T_k(2^{1+1/k}-1)
\\
&=T_{2k}(2^{1/(2k)}),
\end{align*}
where the duplication formula $T_k(2x^2-1)=T_{2k}(x)$ was applied.

The duplication formula in fact allows one to write the very same polynomial~$Q(z)$ in the form
$$
Q(z)=\pm (-z)^{n/2}\, T_n\big(2^{1/n - 1} {\textstyle\sqrt{2 - (z+z^{-1})}}\,\big),
$$
and this formula gives the wanted polynomial, monic and of degree $n$, for $n$ of any parity.
If we set $k=\lfloor(n+1)/2\rfloor$, the zeroes $z_1,\dots,z_n$ of $Q(z)$ pair as before, $z_{n-j}=\ol z_j=z_j^{-1}$ for $j=1,\dots,k$, with the two zeroes merging into one, $z_{(n+1)/2}=1$ for $j=k$ when $n$~is odd,
so that $2^{1/n - 1} \sqrt{2- \smash[b]{(z_j+z_j^{-1})}}$ for $j=1,\dots,k$ are precisely the $k$ real zeroes of the polynomial $T_n(x)$ on the interval $0\le x<1$.
This leads to the estimate
$$
\max_{z\in[z_n,z_1]}|Q(z)|=\max_{|z|=1}|Q(z)|=|Q(-1)|=T_n(2^{1/n})
$$
for both even and odd values of $n$.

Finally, we remark that the \emph{uniqueness} of $Q(z)$, up to rotation, follows from the extremal properties of the Chebyshev polynomials.
\end{proof}

\begin{proof}[Proof of Theorem~\textup{\ref{th2}}]
For this part we cast the Chebyshev polynomial $T_n(x)$ in the form
$$
T_n(x)
=\frac{(x+\sqrt{x^2-1})^n+(x-\sqrt{x^2-1})^n}2
=\frac{x^n}2\cdot
\textstyle\bigl((1+\sqrt{1-x^{-2}})^n+(1-\sqrt{1-x^{-2}})^n\bigr)
$$
leading to
$$
T_n(2^{1/n})
=\textstyle(1+\sqrt{1-e^{-\nu^2}})^n+(1-\sqrt{1-e^{-\nu^2}})^n
$$
in the notation $\nu=\sqrt{(\log4)/n}$.
Since
\begin{align*}
{\textstyle\sqrt{1-e^{-\nu^2}}}
&=\biggl(\sum_{k=1}^\infty\frac{(-1)^{k-1}\nu^{2k}}{k!}\biggr)^{\! 1/2}
=\nu\biggl(1+\sum_{k=2}^\infty\frac{(-1)^{k-1}\nu^{2k-2}}{k!}\biggr)^{\! 1/2}
\\
&=\nu\cdot\bigg(1 - \frac14\nu^2 + \frac5{96}\nu^4 - \frac1{128}\nu^6 + \frac{79}{92160}\nu^8 - \frac3{40960}\nu^{10}+O(\nu^{12})\bigg),
\end{align*}
we conclude that the term $(1-\sqrt{1-e^{-\nu^2}})^n=O(\eps^n)$ for any choice of positive $\eps<1$, hence
\begin{equation*}
\big(T_n(2^{1/n})\big)^{1/n}
=(1+\textstyle\sqrt{1-e^{-\nu^2}})\cdot\big(1+O(\eps^n)\big),
\end{equation*}
and the required asymptotics follows.
\end{proof}

\section{Speculations}
\label{sec4}

Dimitrov's estimate $t(K)\ge1$ for the capacity of the hedgehog $K=K(\beta_1,\dots,\beta_n)$ assigned to a polynomial in $\mathbb Z[x]$ is not necessarily sharp, and one would rather expect to have $t(K)\ge t$ for some $t>1$.
By replacing the polynomial in the proof of Theorem~\ref{th1} with
$$
Q(z)=\pm (-z)^{n/2}\, T_n\big(2^{1/n - 1}t {\textstyle\sqrt{2 - (z+z^{-1})}}\,\big)
$$
and assuming (or, better, believing!) that the corresponding minimum in Question~\ref{qu2} is indeed attained in the case when all but one factors are equal to~1, we conclude that the minimum is equal to $\bigl(T_n(2^{1/n}t)\bigr)^{\! 1/n}$.
The asymptotics of the Chebyshev polynomials then converts this result into the answer
$$
\inf_{\substack{n=1,2,\dots\\ K = K(\beta_1, \dots, \beta_n)\\ t(K) \ge t}}\prod_{j=1}^n\max\{1,|\beta_j|\}
\ge t+\textstyle\sqrt{t^2-1}
$$
to the related version of Question~\ref{qu1}.
This is slightly better, when $t>1$, than the trivial estimate of the infimum by $t$ from below.

In another direction, one may try to associate hedgehogs $K$ to polynomials in a different (more involved!) way, to achieve some divisibility properties for the Hankel determinants $A_k$ that appear in estimation $t(K)\ge\limsup_{k\to\infty}|A_k|^{1/k^2}$ of the capacity on the basis of P\'olya's theorem.
Such an approach has potential to lead to some partial (`Dobrowolski-type') resolutions of Lehmer's problem.
Notice however that the bound for $t(K)$ in P\'olya's theorem is not sharp:
numerically, the Hankel determinants $A_k=\det_{0\le i,j<k}(a_{i+j})$ constructed on (Dimitov's) irrational series
\begin{align*}
\sum_{k=0}^\infty a_kx^k
&=\textstyle\sqrt{(x-\alpha_1^2)(x-\alpha_2^2)(x-\alpha_3^2)\,(x-\alpha_1^4)(x-\alpha_2^4)(x-\alpha_3^4)}
\\
&=\textstyle\sqrt{(1-x+2x^2-x^3)(1+3x+2x^2-x^3)}\in\mathbb Z[[x]]
\end{align*}
for Smyth's polynomial $x^3-x-1=(x-\alpha_1)(x-\alpha_2)(x-\alpha_3)$ satisfy $|A_k|\le C^k$ for some $C<2.5$ and all $k\le150$, so that it is likely that $\limsup_{k\to\infty}|A_k|^{1/k^2}=1$ in this case.

\medskip\noindent
\textbf{Acknowledgements.}
The third author thanks Yuri Bilu and Laurent Habsieger for inspirational conversations on Lehmer's and Schnizel--Zassenhaus problems.

\end{document}